\documentclass[11pt,reqno]{amsart}
\usepackage{epic}
\usepackage{arydshln, bm, amsmath}
\usepackage{amsmath}
\usepackage{amstext}
\usepackage{amssymb}
\usepackage{color}
\usepackage{wrapfig}
\usepackage{tikz}

\renewcommand{\l}{\lambda}

\newcommand{\setZ}{{\mathbb{Z}}}
\newcommand{\setN}{{\mathbb{N}}}

\newcommand{\maj}{{\rm{maj}}}
\newcommand{\vmr}{{\rm{vmr}}}
\newcommand{\cp}{{\rm{Com}}}
\newcommand{\Acal}{{\mathcal{A}}}
\newcommand{\Dcal}{{\mathcal{D}}}
\newcommand{\Ecal}{{\mathcal{E}}}
\newcommand{\Fcal}{{\mathcal{F}}}
\newcommand{\Lcal}{{\mathcal{L}}}
\newcommand\qbin[2]{{\left[\begin{matrix} #1 \\ #2 \end{matrix} \right]}}

 \theoremstyle{plain}
\newtheorem{theorem}{Theorem}[section]
\newtheorem{lemma}[theorem]{Lemma}
\newtheorem{corollary}[theorem]{Corollary}
\newtheorem{proposition}[theorem]{Proposition}
\newtheorem{definition}[theorem]{Definition}
\newtheorem{example}[theorem]{Example}
\newtheorem{remark}[theorem]{Remark}

\begin{document}

\title[Enumeration of partitions with prescribed successive rank parity blocks]
{Enumeration of partitions with prescribed successive rank parity blocks}

\author{Seunghyun Seo}
\address{Department of Mathematics Education, Kangwon National University,
Chuncheon, Kangwon 24341, The Republic of Korea} \email{shyunseo@kangwon.ac.kr}

\author{Ae Ja Yee}
\address{Department of Mathematics, The Pennsylvania State University,
University Park, PA 16802, USA} \email{auy2@psu.edu}


\maketitle

\footnotetext[1]{The first author was partially supported by a research grant of Kangwon National University in 2015.} \vspace{0.5in}
\footnotetext[2]{The second author was partially supported by a grant ($\#$280903) from the Simons Foundation.} \vspace{0.5in}
\footnotetext[3]{2010 AMS Classification Numbers: Primary, 05A17; Secondary, 11P81.}

\noindent{\footnotesize{\bf Abstract.}
Successive ranks of a partition, which were introduced by Atkin,  are the difference of the $i$-th row and the $i$-th column in the Ferrers graph.  Recently, in the study of singular overpartitions, Andrews revisited successive ranks and parity blocks.
Motivated by his work, we investigate partitions with prescribed successive rank parity blocks.  The main result of this paper is the generating function of partitions with exactly $d$ successive ranks and $m$ parity blocks. }


\section{Introduction}

In 1944,  defining the rank of a partition as the largest part minus the number of parts, F. Dyson conjectured that the rank statistic would account for the following Ramanujan partition congruences combinatorially \cite{dyson1944}:
\begin{align*}
p(5n+4)\equiv 0 \pmod{5},\\
p(7n+5)\equiv 0 \pmod{7},
\end{align*}
where $p(n)$ denotes the number of partitions of $n$. Dyson's conjecture was proved by A. O. L. Atkin and H. P. F. Swinnerton-Dyer \cite{atkinswinnertondyer66}. In search of the crank statistic   for
\begin{equation*}
p(11n+6)\equiv 0 \pmod{11}
\end{equation*}
whose existence was conjectured by Dyson,
 Atkin introduced successive ranks  and showed how they would replace the rank statistic in various algebraic expressions \cite{atkin66}.

Developing a new sieve method for counting partitions, G. E. Andrews investigated generating functions related to successive ranks \cite{gea50}, and subsequently Andrews' work was generalized by D. M. Bressoud \cite{bressoud1980sieve}.

In \cite{gea50} and \cite{bressoud1980sieve}, the concept of oscillations of the successive ranks were introduced and used to study Rogers-Ramanujan type partition identities, and further study was made on variations of successive ranks in  \cite{gea123}.  Recently, in the study of singular overpartitions \cite{gea303}, Andrews revisited successive ranks and parity blocks.

The main purpose of this article is to study partitions with prescribed successive ranks and their parity blocks. It turns out that the enumeration of such partitions involves counting a certain type of plane partitions.

A partition may be represented by a Ferrers graph \cite{geapartitions}.  For instance, the partition $7+5+5+3+2+2+1$ has the representation:
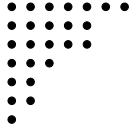
\begin{figure}[ht]
\begin{tikzpicture}[scale=1]
\foreach \x in {0,...,6}
	\filldraw (\x*.25, 0) circle (.5mm);
\foreach \x in {0,...,4}
	\filldraw (\x*.25, -.25) circle (.5mm);
	\foreach \x in {0,...,4}
	\filldraw (\x*.25, -.5) circle (.5mm);
	\foreach \x in {0,...,2}
	\filldraw (\x*.25, -.75) circle (.5mm);
	\foreach \x in {0,...,1}
	\filldraw (\x*.25, -1) circle (.5mm);
\foreach \x in {0,...,1}
	\filldraw (\x*.25, -1.25) circle (.5mm);
\foreach \x in {0,...,0}
	\filldraw (\x*.25, -1.5) circle (.5mm);
\end{tikzpicture}
\end{figure}

\noindent The successive ranks of a partition are defined along the main diagonal of the Ferrers graph, namely the $i$-th rank is the number of dots in row $i$ minus  the number of dots in column $i$.

Meanwhile, in the Ferrers graph, by reading off rows to the right of (resp. columns  below)  the main diagonal and putting their sizes on the top row (resp. bottom row), we can represent a partition $\lambda$ of $n$ as a two-rowed array \cite{gea109, yee12}:
 \begin{equation*}
\left( \begin{array}{cccc} x_1 & x_2 & \cdots & x_d
\\
y_1 & y_2 & \cdots & y_d
\end{array} \right),
\end{equation*}
where $d$ is the number of dots in the main diagonal, $\sum_{i=1}^{d}(x_i +y_i+1)=n$,  $x_1>x_2>\cdots>x_d \ge0$, and
$y_1>y_2>\cdots>y_d \ge0$. This is called a Frobenius symbol of $\lambda$.  The Frobenius symbol of the partition $7+5+5+3+2+2+1$ is
\begin{equation*}
\left(\begin{array}{ccc} 6 & 3& 2\\ 6 & 4 & 1\end{array} \right).
\end{equation*}
From the construction, we can easily see that $x_i-y_i$ is the $i$-th rank.

Throughout this paper, we will use Frobenius symbols to represent partitions.  For a partition $\lambda$,
a column $\begin{matrix} x_i \\ y_i \end{matrix}$ of $\l$ is said to be positive if $x_i-y_i \ge 1$ and {negative} if $x_i-y_i \le 0$.  If two columns are both positive or both negative, we shall say that they have the same parity.  We now divide $\l$ into {parity blocks}.  These are sets of contiguous columns maximally extended to the right, where all the entries have the same parity.   We shall say that a block is positive (resp. negative) if it contains positive (resp. negative) columns. The  Frobenius symbol above has two parity blocks, negative and positive as follows:
\begin{equation*}
\left(\begin{array}{cc|c}6 & 3 & 2\\ 6 & 4 & 1\end{array}\right),
\end{equation*}
where the vertical bar divides the blocks.

For positive integers $d$ and  $m$ with $d\ge m$, let $a^+_{m}(n;d)$ and $a^-_{m}(n;d)$ be the numbers of partitions of $n$ into exactly $d$ columns and $m$ parity blocks, where the last block is positive and negative, respectively. The main result of this paper is stated in the following theorem.

\begin{theorem} \label{thm:main}
For $d\ge m\ge 1$,
\begin{equation} \label{eq:mainP}
\sum_{n=1}^{\infty} a^+_{m}(n;d) q^n =
\frac{q^{d^2+d+\binom{m}{2}}}{(q;q)_{2d}} \frac{1-q^m}{1-q^d} \qbin{2d}{d+m}
\end{equation}
and
\begin{equation} \label{eq:mainN}
\sum_{n=1}^{\infty} a^-_{m}(n;d) q^n =
\frac{q^{d^2+\binom{m}{2}}}{(q;q)_{2d}} \frac{1-q^m}{1-q^d} \qbin{2d}{d+m}.
\end{equation}
\end{theorem}
Here and in the sequel, we employ the customary $q$-series notation:
\begin{align*}
(a;q)_n&:=(1-a)(1-aq)\cdots (1-aq^{n-1}),\\
(a;q)_{\infty}&:= \lim_{n\to \infty} (a;q)_n, \quad |q|<1,\\
\qbin{n}{k}&:=
\begin{cases}
\frac{(q;q)_n}{(q;q)_{k}(q;q)_{n-k}}, & \text{if $n\ge k\ge 0$},\\
0, & \text{otherwise}.
\end{cases}
\end{align*}
Surprisingly, when $m=1$, the right-hand side of \eqref{eq:mainP} involves a $q$-analogue of the Catalan numbers, namely
\begin{equation}
\frac{(1-q)}{(1-q^d)} \qbin{2d}{d+1}. \label{qcatalan}
\end{equation}
These $q$-Catalan numbers enumerate the major index of Catalan words \cite{furlingerhofbauer}. Thus, to prove Theorem~\ref{thm:main},  we will  generalize Catalan words in Section~\ref{sec:PP}.


For a positive integer $m$, let $a^+_m(n)$  and $a^-_m(n)$  be the numbers of partitions of $n$ into $m$ parity blocks such that the last block is positive and negative, respectively. 
Then their generating functions  are as follows.

\begin{theorem} \label{thm1.2}
For $m\ge 1$, 
\begin{equation*}
\sum_{n=1}^{\infty} a^+_m(n) q^n =(-1)^{m}+
\frac{(-1)^{m-1}}{(q;q)_{\infty}} \sum_{l=0}^{m-1}(-1)^{l}
q^{\frac{l(3l+1)}{2}}(1-q^{2l+1})
\end{equation*}
and
\begin{equation*}
\sum_{n=1}^{\infty} a^-_m(n) q^n =(-1)^{m}+
\frac{(-1)^{m-1}}{(q;q)_{\infty}} \sum_{l=0}^{m-1}(-1)^l
q^{\frac{l(3l-1)}{2}}(1-q^{4l+2}).
\end{equation*}
\end{theorem}

\bigskip

Note that
\begin{align*}
a^+_m(n) &=\sum_{d= 1}^{\infty} a^+_{m}(n;d),\\
a^-_m(n) &=\sum_{d= 1}^{\infty} a^-_{m}(n;d).
\end{align*}
Hence, the following corollary is immediate by Theorems~\ref{thm:main} and \ref{thm1.2}.


\begin{corollary} \label{euler}
For $m\ge1$,
\begin{equation}\label{eq:AM}
\frac{1}{(q;q)_{\infty}} \sum_{l=0}^{m-1}(-1)^{l}
q^{\frac{l(3l+1)}{2}}(1-q^{2l+1})=
1+(-1)^{m-1}\sum_{d=1}^{\infty}
\frac{q^{d^2+d+\binom{m}{2}}}{(q;q)_{2d}} \frac{1-q^m}{1-q^d}\qbin{2d}{d+m}
\end{equation}
and
\begin{equation*}
\frac{1}{(q;q)_{\infty}} \sum_{l=0}^{m-1}(-1)^{l}
q^{\frac{l(3l-1)}{2}}(1-q^{4l+2})=
1+(-1)^{m-1}\sum_{d=1}^{\infty}
\frac{q^{d^2+\binom{m}{2}}}{(q;q)_{2d}} \frac{1-q^m}{1-q^d}\qbin{2d}{d+m}.
\end{equation*}
\end{corollary}

Corollary~\ref{euler} suggests some connection to the work of Andrews and M. Merca on the truncated pentagonal number theorem \cite{gea288}.  Let $M_m(n)$ be the number of partitions of $n$ in which $m$ is the least integer that is not a part and there are more parts $> m$ than there are $< m$. Andrews and Merca showed that the generating function for $M_m(n)$ is the truncated version of the Euler pentagonal number theorem, namely the left-hand side of the identity \eqref{eq:AM}.
Thus we see that $a_{m}^{+}(n)=M_m(n)$.


%

For a positive integer $d$, let $a^+(n;d)$  and $a^-(n;d)$  be the numbers of partitions of $n$ into $d$ columns such that the last block is positive and negative, respectively. 
Then their generating functions are as follows.

\begin{theorem} \label{thm1.4}
For $d\ge 1$,
\begin{equation*}
\sum_{n=1}^{\infty} a^+(n;d) q^n = \frac{q^{d^2+d}}{(q;q)_d^2 \, (1+q^d)}
\end{equation*}
and
\begin{equation*}
\sum_{n=1}^{\infty} a^-(n;d) q^n = \frac{q^{d^2}}{(q;q)_d^2 \,(1+q^d)}.
\end{equation*}
\end{theorem}

Note that
\begin{align*}
a^+(n;d) &:=\sum_{m=1}^d a^+_{m}(n;d),\\
a^-(n;d) &:=\sum_{m=1}^d a^-_{m}(n;d).
\end{align*}
The following corollary follows from Theorems~\ref{thm:main} and \ref{thm1.4}.
\begin{corollary}
For $d\ge 1$,
\begin{equation*}
\sum_{m=1}^{d} q^{\binom{m}{2}}(1-q^m) \qbin{2d}{d+m}=\frac{1-q^d}{1+q^d} \qbin{2d}{d}.
\end{equation*}

%
\end{corollary}

This paper is organized as follows. In Section~\ref{sec2}, necessary definitions on lattice paths are reviewed and relevant lemmas are proved. In Section~\ref{sec:PP}, preliminaries on posets and partitions associated with posets are given. Our main result in Theorem~\ref{thm:main}
will be proved in Section~\ref{sec4}, and Theorems~\ref{thm1.2} and \ref{thm1.4} will be proved in Section~\ref{sec5}.

\section{Lattice paths}\label{sec2}

Let $S$ be a subset of $\setZ\times \setZ$.  A lattice path in $\setZ\times \setZ$ of length $m$ with steps in $S$ is a sequence $v_0,v_1,\ldots, v_m $ in  $\setZ\times \setZ$ such that each consecutive  difference $v_i-v_{i-1}$ lies in $S$. We say that the path goes from $v_0$ to $v_m$.  A ballot path from $(0,0)$ to $(s+t,s-t)$ is a lattice path with steps $(1,1)$ and $(1,-1)$, with the additional condition that it never passes below the $x$-axis.  A Dyck path of length $2n$ is a ballot path with $s=t=n$.

For a ballot path $D=v_0,v_1,\ldots, v_m$, if $v_i-v_{i-1}=(1,-1)$ and $v_{i+1}-v_{i}=(1,1)$, then $v_i$ is called a valley.  If $v_i$ is a valley on the $x$-axis, then it is called a return. If $v_i$ is a return,
then it may be overlined and called a marked return.

For convenience, we also denote steps $(1,1)$ and $(1,-1)$ by $\rm u$ and $\rm d$, respectively, and then write a path  as a word of steps $\rm u$ and $\rm d$ accordingly.  Here, $\rm u$ and $\rm d$ stand for up-step and down-step, respectively.  If there is a marked return, then we put a vertical bar between the steps $\rm d$ and $\rm u$ associated with the return.  For instance, let $D=(0,0), (1,1), (2,0), (3,1), \overline{(4,0)}, (5,1)$. Then it is a ballot path from $(0,0)$ to $(5,1)$ and the corresponding word is $\rm udud|u$.

For nonnegative integers $s$, $t$, and $r$, let $\Dcal_{s,t}^{(r)}$ be the set of ballot paths from $(0,0)$ to $(s+t,s-t)$ with at least $r$ marked returns.
For a path $D \in \Dcal_{s,t}^{(r)}$, we define $\maj(D)$ by
\begin{equation*}
\maj(D):=\sum_{v_i} x_i,
\end{equation*}
where the sum is over all valleys $v_i$ of $D$ and $x_i$ is the $x$-coordinate of $v_i$.
For a path $D \in \Dcal_{s,t}^{(r)}$, we define a weight $\vmr(D)$ by
\begin{equation*}
\vmr(D):=\maj(D) -\frac{1}{2} \sum_{\overline{v_i}} x_i,
\end{equation*}
where the sum is over all marked returns $\overline{v_i}$ of $D$, and $x_i$ is the $x$-coordinate of $\overline{v_i}$.

\begin{lemma} \label{lem:maD}
For nonnegative integers $s$, $t$ with $s>t$, and $r$,
\begin{equation} \label{eqn:vmrD}
\sum_{D \in \Dcal_{s,t}^{(r)}} q^{\vmr(D)} = q^{\binom{r+1}{2}}\qbin{s+t}{s+r}.
\end{equation}
\end{lemma}
\begin{proof}
We will prove by induction on $s+t$.
If $s+t=1$, then $s=1$ and $t=0$. Thus there is only one ballot path from $(0,0)$ to $(1,0)$ and
 we have
$$
\sum_{D\in \Dcal_{1,0}^{(r)}} q^{\vmr(D)} = \delta_{r,0}= q^{\binom{r+1}{2}}\qbin{1}{1+r}.
$$
Thus \eqref{eqn:vmrD} holds for $s+t=1$.

For a positive integer $n\ge 2$, we assume that \eqref{eqn:vmrD} holds for all $s+t<n$.
We now consider $s+t=n$. Let $D$ be a path from $\Dcal^{(r)}_{s,t}$, $D'$ be the path resulting from removing the last step in $D$, and $D''$ be the path resulting from removing the last two steps in $D$, namely
\begin{align}
D&=v_0, v_1,\ldots, v_{s+t-2}, v_{s+t-1}, v_{s+t},  \label{D}\\
D' & =v_0,v_1,\ldots, v_{s+t-2}, v_{s+t-1},   \label{D'} \\
D'' &=v_0,v_1,\ldots, v_{s+t-2}.  \label{D''}
\end{align}
Note that $v_{s+t}=(s+t, s-t)$.  We divide $s-t$ into two cases: $s-t\ge 2$ and $s-t=1$.

\begin{itemize}
\item Case 1: $s-t\ge2$.
\begin{itemize}
\item
Suppose that the last step in $D$ is $\rm d$.  Then $v_{s+t-1}=(s+t-1, s-t+1)$, which  cannot be a valley, so
$D' \in \Dcal^{(r)}_{s,t-1}$ and
$$
\vmr(D)=\vmr(D').
$$
By the induction hypothesis, we see that the weight sum of such paths equals
$$
q^{\binom{r+1}{2}}\qbin{s+(t-1)}{s+r}.
$$

\item
Suppose that the last two steps in $D$ are $\rm du$. Then $v_{s+t-1}=(s+t-1, s-t-1) $ is a valley. Since $s-t\ge 2$, $v_{s+t-1}$ cannot be a point on the $x$-axis, which yields that it cannot be a return. Since $v_{s+t-2}=(s+t-2, s-t)$,  $D'' \in \Dcal^{(r)}_{s-1,t-1}$ and
$$
\vmr(D)=q^{s+t-1} \vmr(D'').
$$
Hence, by the induction hypothesis, we see that the weight sum of such paths equals
$$
q^{s+t-1} q^{\binom{r+1}{2}}\qbin{(s-1)+(t-1)}{(s-1)+r}.
$$

\item Lastly, suppose that the last two steps in $D$ are $\rm uu$. Then $v_{s+t-1}$ cannot be a valley, so
$$
\vmr(D)=\vmr(D').
$$
Also, note that $D'$ is a path from $(0,0)$ to $(s+t-1,s-t-1)$ whose last step is $\rm u$, so the set of all such paths $D'$ is the set of all paths from $(0,0)$ to $(s+t-1,s-t-1)$ minus the set of all paths from $(0,0)$ to $(s+t-1,s-t-1)$ whose last step is $\rm d$.  If a path from $(0,0)$ to $(s+t-1, s-t-1)$ ends with $\rm d$, then the path cannot have a valley at the second last lattice point $(s+t-2, s-t)$.  Thus the weight of $D''$ remains the same. 
By the induction hypothesis,  we see that the weight sum of such paths whose last two steps are $\rm uu$ equals
$$
q^{\binom{r+1}{2}}\qbin{(s-1)+t}{(s-1)+r} - q^{\binom{r+1}{2}}\qbin{(s-1)+(t-1)}{(s-1)+r} = q^{\binom{r+1}{2}+t-r}\qbin{s+t-2}{s+r-2},
$$
where the first summand on the left hand side accounts for the generating function for paths $D'$ from $(0,0)$ to $(s+t-1,s-t-1)$, and the second summand accounts for the generating function for paths $D''$ from $(0,0)$ to $(s+t-2,s-t)$.
\end{itemize}
Therefore,
\begin{align*}
\sum_{D\in \Dcal_{s,t}^{(r)}} q^{\vmr(D)} &= q^{\binom{r+1}{2}}\left( \qbin{s+t-1}{s+r}+q^{s+t-1}\qbin{s+t-2}{s+r-1}+
q^{t-r}\qbin{s+t-2}{s+r-2}\right) \\
&= q^{\binom{r+1}{2}}\qbin{s+t}{s+r}.
\end{align*}

\item Case 2: $s-t=1$. Recall the paths $D$, $D'$, and $D''$ defined in \eqref{D}, \eqref{D'}, and \eqref{D''}.
\begin{itemize}
\item
Suppose that the last step in $D$ is $\rm d$. Then $v_{2s-2}=(2s-2, 2)$ cannot be a valley, so
$D' \in \Dcal^{(r)}_{s,s-2}$ and
$$
\vmr(D)=\vmr(D').
$$
By the induction hypothesis, we see that the weight sum of such paths equals
$$
q^{\binom{r+1}{2}}\qbin{s+(s-2)}{s+r}.
$$

\item
Suppose that the last step in $D$ is $\rm u$.  Then the second last step must be $\rm d$, for $s-t=1$, i.e., $v_{2s-2}=(2s-2,0)$, and $D$ is a ballot path.  Thus $v_{2s-2}$ is a return.

If $v_{2s-2}$ is not marked, then $D''\in \Dcal_{s-1,s-2}^{(r)}$ and
$$
 \vmr(D)=q^{2s-2}  \vmr(D'').
$$
Thus, by the induction hypothesis, we see that  the weight sum of  such paths equals
$$
q^{2s-2} q^{\binom{r+1}{2}}\qbin{(s-1)+(s-2)}{(s-1)+r}.
$$
If $v_{2s-2}$ is marked, then $D''\in \Dcal_{s-1,s-2}^{(r-1)}$ and
$$
 \vmr(D)=q^{s-1}  \vmr(D'').
$$
   Then, by the induction hypothesis, the weight sum of such paths equals
$$
q^{s-1} q^{\binom{r}{2}}\qbin{(s-1)+(s-2)}{(s-1)+(r-1)}.
$$
\end{itemize}
Therefore,
\begin{align*}
\sum_{D\in \Dcal_{s,s-1}^{(r)}} q^{\vmr(D)} &= q^{\binom{r+1}{2}}\left( \qbin{2s-2}{s+r}+q^{2s-2}\qbin{2s-3}{s+r-1}+
q^{s-1-r}\qbin{2s-3}{s+r-2}\right) \\
& = q^{\binom{r+1}{2}}\qbin{2s-1}{s+r}.
\end{align*}
\end{itemize}
This completes the proof.
\end{proof}

\begin{example}
There are $5$ paths in $\Dcal^{(1)}_{3,2}$ as follows:
\begin{alignat*}{2}
D_1&={\rm uudd|u} \quad&& \vmr(D_1)= (4) - \frac12 (4) =2 \\
D_2&={\rm ud|uud} \quad&& \vmr(D_2)= (2) - \frac12 (2) =1 \\
D_3&={\rm ud|udu} \quad&& \vmr(D_3)= (2+4) - \frac12 (2) =5 \\
D_4&={\rm udud|u} \quad&& \vmr(D_4)= (2+4) - \frac12 (4) =4 \\
D_5&={\rm ud|ud|u} \quad&& \vmr(D_5)= (2+4) - \frac12 (2+4) =3.
\end{alignat*}
Thus
$$
\sum_{D \in \Dcal_{3,2}^{(1)}} q^{\vmr(D)}= q^1+q^2+q^3+q^4+q^5
= q^{\binom{1+1}{2}}\qbin{3+2}{3+1}.
$$
\end{example}

We now compute the generating function for Dyck paths from $(0,0)$ to $(2s,0)$ with at least $r$ marked returns.

\begin{lemma}\label{lem:s=t}
For a positive integer $s$ and a nonnegative integer $r$,
\begin{equation} \label{eqn:s=t}
\sum_{D \in \Dcal_{s,s}^{(r)}} q^{\vmr(D)} = q^{\binom{r+1}{2}}\qbin{2s-1}{s+r}.
\end{equation}
\end{lemma}

\proof
Since $v_{2s}=(2s,0)$,  each lattice path $D=v_0,v_1, \ldots, v_{2s}$ in $\Dcal_{s,s}^{(r)}$ must end with a down step $\rm d$.
Thus, $v_{2s-1}$ cannot be a valley, from which it follows that
\begin{equation*}
{\vmr(D)}={\vmr(D')}.
\end{equation*}
In addition, the number of marked returns remains the same, so  $D'=v_0, v_1, \ldots, v_{2s-1}$ is in $\Dcal_{s,s-1}^{(r)}$.
 Therefore, the proof follows from Lemma~\ref{lem:maD}.
\endproof

Let $\Ecal_{s}^{(r)}$ be the set of Dyck paths from $(0,0)$ to $(2s,0)$ with exactly $r$ marked returns.

\begin{corollary}\label{coro:maE}
For a positive integer $s$ and a nonnegative integer $r$,
\begin{equation*}
\sum_{D \in \Ecal_{s}^{(r)}} q^{\vmr(D)} = q^{\binom{r+1}{2}}\frac{1-q^{r+1}}{1-q^s}
\,\qbin{2s}{s+r+1}.
\end{equation*}
\end{corollary}

\begin{proof}
Note that $\Ecal_{s}^{(r)}=\Dcal_{s,s}^{(r)} \setminus \Dcal_{s,s}^{(r+1)}$.  Hence,
\begin{align*}
\sum_{D \in \Ecal_{s}^{(r)}} q^{\vmr(D)} &=\sum_{D \in \Dcal_{s,s}^{(r)}} q^{\vmr(D)}
-\sum_{D \in \Dcal_{s,s}^{(r+1)}} q^{\vmr(D)} \\
&=q^{\binom{r+1}{2}} \qbin{2s-1}{s+r} -q^{\binom{r+2}{2}} \qbin{2s-1}{s+r+1}\\
&=q^{\binom{r+1}{2}}\frac{1-q^{r+1}}{1-q^s} \, \qbin{2s}{s+r+1},
\end{align*}
where the second equality follows from Lemma~\ref{lem:s=t}.
\end{proof}

\begin{remark}
From Lemma~\ref{lem:maD},  $\left|\Dcal_{s,t}^{(r)}\right|=\binom{s+t}{s+r}$ for all $s>t \ge r \ge 0$. Also, it is trivial from the definition that $\Dcal_{s,t}^{(r+1)} \subset \Dcal_{s,t}^{(r)}$.
If we set $s-1=t+1=n$ (resp. $s=t+1=n$), then
we immediately get
$$
\binom{2n}{n+r+2} < \binom{2n}{n+r+1}. \quad \left(\mbox{resp.}~\binom{2n-1}{n+r+1} < \binom{2n-1}{n+r}.\right)
$$
Thus we have
$$\binom{N}{j+1}<\binom{N}{j}$$
for all $\lceil N/2 \rceil < j \le N$. This is an injective proof of the unimodality of binomial coefficients, where the injection is the inclusion map.
\end{remark}

\section{Posets and partitions} \label{sec:PP}
In this section, we first recollect some preliminary definitions and notation for posets and partitions from \cite[\S 3.13 and \S 3.15]{sta1}.

For a poset $P$ of size $n$, a labeling $\omega$ is a bijection
$\omega:P\to\{1,\ldots,n\}$.  We say a labeling $\omega$ of $P$ is natural if $ \omega(x)<\omega(y)$ whenever $x<_P y$.

A linear extension $\sigma$ of $P$ is an order-preserving bijection $\sigma:P\to\{1',\ldots,n' \}$, i.e., $\sigma(x)<\sigma(y)$ whenever $x<_P y$.
Given a natural labeling $\omega$, we identify $\sigma$ 
with a permutation $w=\omega(\sigma^{-1}(1'))\cdots\omega(\sigma^{-1}(n'))$.
Let ${\Lcal}(P,\omega)$ be the set of linear extensions of $P$, regarded as permutations of the set $\{1,\ldots,n\}$.

Let $\setN$ be the set of nonnegative integers.  For a poset $P$ with a labeling $\omega$, a $(P,\omega)$-partition $\pi$ is a function $\pi:P\to\setN$ satisfying the following conditions.
\begin{itemize}
\item If $x\le_{P} y$, then $\pi(x) \ge \pi(y)$ (order-reversing);
\item if $x<_{P} y$ and $\omega(x)>\omega(y)$, then $\pi(x)>\pi(y)$.
\end{itemize}

We denote by $\Acal(P,\omega)$ the set of $(P,\omega)$-partitions.  If $\omega$ is natural, then a $(P,\omega)$-partition is just an order-reversing map from $P$ to $\setN$. 
For more details, see \cite[\S 3.13 and \S 3.15]{sta1}.

\subsection{The poset $S_\beta$}

\begin{definition} \label{def:ri}
For a composition $\beta=(b_1,\ldots,b_m)$, we define $r_i(\beta)$ by
$$r_i(\beta):=b_1+\cdots+b_i$$
for $1\le i \le m$ with $r_0(\beta)=0$. If $\beta$ is clear in the context, we will use $r_i$ for $r_i(\beta)$.
\end{definition}

For a composition $\beta=(b_1,\ldots,b_m)$, let $S_\beta$ be the poset
$$
\bigcup_{l=1}^{m}\{(i,j): i,j\in \setN,~ l\le i \le l+1,~r_{l-1}+1 \le j \le r_l \},
$$
where
 $(i_1,j_1)\le_{S_{\beta}} (i_2,j_2)$ if $i_1\le i_2$ and $j_1\le j_2$. In other words, 
\begin{equation}
S_{\beta} \cong  ({\bm 2}\times {\bm b_{\bm 1}})\oplus \cdots \oplus ({\bm 2}\times {\bm b_{\bm m}}), \label{posetsbeta}
\end{equation}
where
${\bm n}$ is the chain of $n$  elements, i.e., a linearly ordered poset of $n$ elements for a positive integer $n$, $P \times Q$ is the direct product of two posets $P$ and $Q$ with $(p_1, q_1)\le (p_2, q_2)$ in $P\times Q$ whenever $p_1\le_P p_2$ and $q_1\le_Q q_2$, and $P\oplus Q$ is the ordinal sum of $P$ and $Q$ with $x\le y$  in the disjoint union $P \cup Q$ whenever $x\le_P y$, or $x\le_Q y$, or $x\in P$ and $y\in Q$.  For more details, see \cite[\S 3.2]{sta1}.

The diagram of $S_{(2,3,1,2)}$ is given in  Figure~\ref{fig1}, where circles denote the elements of the poset and inequality sign direction indicates the order between elements.

 \begin{center}{
 \begin{figure}
\begin{tikzpicture}[scale=1.5]
\draw[thick, black] (0.0, 0.0) --(0.5, 0.0) -- (0.5, -0.5) -- (0.0, -0.5) -- (0.0, 0.0);
\draw[thick, black] (0.5, -0.5) -- (1.0, -0.5) ;
\draw[thick, black] (1.0, -0.5) -- (2.0, -0.5) -- (2.0, -1.0) -- (1.0, -1.0) -- (1.0, -0.5);
\draw[thick, black] (2.0, -1.0) -- (2.5, -1.0) ;
\draw[thick, black] (2.5, -1.0) -- (2.5, -1.5) ;
\draw[thick, black] (2.5, -1.5) -- (3.0, -1.5) ;
\draw[thick, black] (3.0, -1.5) -- (3.5, -1.5) -- (3.5, -2.0) -- (3.0, -2.0) -- (3.0, -1.5);

\foreach \x in {0,1}
\filldraw[thick, black] (\x*0.5+1.5, -0.5) --(\x*0.5+1.5, -1.0)  ;

	\foreach \x in {0,...,1}
	\filldraw (\x*0.5, 0.0) circle (.5mm);
	\foreach \x in {0,...,4}
	\filldraw (\x*0.5, -0.5) circle (.5mm);
	\foreach \x in {2,...,5}
	\filldraw (\x*0.5, -1.0) circle (.5mm);
	\foreach \x in {5,...,7}
	\filldraw (\x*0.5, -1.5) circle (.5mm);
	\foreach \x in {6,...,7}
	\filldraw (\x*0.5, -2.0) circle (.5mm);

         \foreach \x in {0,...,0}
	\filldraw (\x*0.5+0.3, -0.015) node{$\leftarrow$};
	\foreach \x in {0,...,3}
	\filldraw (\x*0.5+0.3, -0.515) node{$\leftarrow$};
	\foreach \x in {2,...,4}
	\filldraw (\x*0.5+0.32, -1.015) node{$\leftarrow$};
	\foreach \x in {5,...,6}
	\filldraw (\x*0.5+0.3, -1.51) node{$\leftarrow$};
	\foreach \x in {6,...,6}
	\filldraw (\x*0.5+0.3, -2.02) node{$\leftarrow$};
	
	\foreach \x in {0,...,1}
	\filldraw (\x*0.5+0.0, -0.3) node{$\uparrow$};
	\foreach \x in {2,...,4}
	\filldraw (\x*0.5+0.0, -0.8) node{$\uparrow$};
	\foreach \x in {5,...,5}
	\filldraw (\x*0.5+0.0, -1.3) node{$\uparrow$};
	\foreach \x in {6,...,7}
	\filldraw (\x*0.5+0.0, -1.8) node{$\uparrow$};
	
	\foreach \x in {1,...,2}
	\filldraw (\x*0.5-0.5,  0.15) node{\tiny{$(1,\x)$}};

	\foreach \x in {1,...,2}
	\filldraw (\x*0.5-0.5,  -0.65) node{\tiny{$(2, \x)$}};

	\foreach \x in {3,...,5}
	\filldraw (\x*0.5-0.5,  -0.35) node{\tiny{$(2, \x)$}};
	
	\foreach \x in {3,...,5}
	\filldraw (\x*0.5-0.5,  -1.15) node{\tiny{$(3, \x)$}};
	
	\foreach \x in {6,...,6}
	\filldraw (\x*0.5-0.5,  -0.85) node{\tiny{$(3, \x)$}};
	
	\foreach \x in {6,...,6}
	\filldraw (\x*0.5-0.5,  -1.65) node{\tiny{$(4, \x)$}};
	
	\foreach \x in {7,...,8}
	\filldraw (\x*0.5-0.5,  -1.35) node{\tiny{$(4, \x)$}};
	
	\foreach \x in {7,...,8}
	\filldraw (\x*0.5-0.5,  -2.15) node{\tiny{$(5, \x)$}};

\end{tikzpicture}
\caption{$S_{(2,3,1,2)}$ }\label{fig1}
\end{figure}
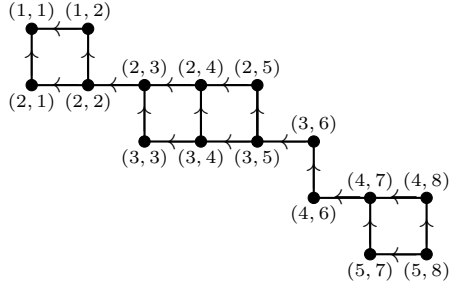}
\end{center}

\begin{lemma} \label{sumdecomp}
Let $P$ and $Q$ be posets of sizes $n_1$ and $n_2$, respectively.  Let $\omega$ be a natural labeling of $P\oplus Q$, and $\omega |_P$ and $\omega |_Q$ be  the natural labelings of $P$ and $Q$, respectively, induced by $\omega$. Then a permutation $w$ in $\Lcal(P\oplus Q, \omega)$ can be expressed uniquely as a concatenation of permutations $w^{(1)}\in \Lcal(P, \omega |_P)$ and $w^{(2)} \in \Lcal(Q, \omega |_Q)$ as follows:
\begin{equation*}
w=\hat{w}^{(1)}\hat{w}^{(2)},
\end{equation*}
where $\hat{w}^{(1)}=w^{(1)}$ and $\hat{w}_j^{(2)}=w^{(2)}_j +n_1$ for $1\le j\le n_2$.
\end{lemma}

\proof
Since $p<_{ P\oplus Q} q$ for any $p \in P$ and $q \in Q$, we note that if  $f: P\oplus Q \to \{1',\ldots, (n_1+n_2)'\}$ is  an order-preserving bijection, then so is each of its restriction maps $f |_P$ and $f |_Q$, where
\begin{align*}
f |_P &: P \to \{1',\ldots, n_1'\} \\
f |_Q &: Q \to \{(n_1+1)',\ldots, (n_1+n_2)' \}.
\end{align*}

Let $\sigma:P\oplus Q \to \{1',\ldots,(n_1+n_2)'\}$ be the linear extension identified with $w$. Since $\omega$ is a natural labeling, $\sigma$ is an order-preserving bijection, and the permutation $w$ is defined as
\begin{align*}
 w_j=\omega \big(\sigma^{-1}(j')\big), \quad 1\le j\le n_1+n_2,
 \end{align*}
 we see that $w_{1} \cdots w_{n_1}$ form a permutation $\hat{w}^{(1)}$ and $w_{n_1+1}\cdots w_{n_1+n_2}$ form a permutation $\hat{w}^{(2)}$, where
\begin{align*}
\hat{w}^{(1)}_j &=\omega |_P\big(\sigma |_P^{-1}(j')\big), \quad 1\le j\le n_1,\\
\hat{w}^{(2)}_j &=\omega |_Q\big(\sigma |_Q^{-1}((n_1+j)')\big), \quad 1\le j\le n_2.
\end{align*}
For $i=1,2$, define $w^{(i)}$ by
\begin{align*}
w^{(i)}_j= \hat{w}^{(i)}_j - n_1\delta_{i,2}, \quad 1\le j\le n_i,
\end{align*}
where $\delta_{i,2}$ is the Kronecker delta.

Since $\hat{w}^{(1)}$ is defined by $\omega |_P$ and $\sigma |_P$,  ${w}^{(1)}$ is in $\Lcal(P, \omega |_P)$. Also, since  $\hat{w}^{(2)}$ is defined by $\omega |_Q$ and $\sigma |_Q$,  ${w}^{(2)}$ is in $\Lcal(Q,  \omega |_Q)$.
 \endproof

\subsection{The natural labeling $\nu$ of $S_{\beta}$ and linear extensions.}
In this subsection, we define a natural labeling $\nu$ of $S_{\beta}$ and characterize linear extensions of $S_{\beta}$ with respect to $\nu$.

Given a composition $\beta=(b_1,\ldots,b_m)$ of $d$, define $\nu:S_{\beta}\to \{1,2,\ldots,2d\}$ by
$$\nu(i,j)=r_{l-1}+j+(i-l)b_{l},$$ where $l$ is an integer such that
$r_{l-1}+1 \le j \le r_l$.  We can easily see that $\nu$ is a natural labeling.

In the rest of this paper, for the poset $S_{\beta}$, we will use the labeling $\nu$ only. Thus $\nu$ will be omitted in definitions and notation.



\begin{lemma} \label{permconc}
Let $w $ be a permutation  in $\Lcal(S_{\beta})$. 
Then $w$ may be expressed as a concatenation of permutations $w^{(i)}$ in $\Lcal({\bm 2}\times{\bm b_{\bm i}})$ as follows:
\begin{equation} \label{permutationw}
w=\hat{w}^{(1)} \,\hat{w}^{(2)} \, \cdots \, \hat{w}^{(m-1)}\, \hat{w}^{(m)},
\end{equation}
where  $\hat{w}^{(i)}_j=w^{(i)}_j+ 2r_{i-1}$ for $1\le j \le 2b_i$.
\end{lemma}

\proof
Use induction on $m$. Then the statement immediately follows from \eqref{posetsbeta} and Lemma~\ref{sumdecomp}. We omit the details.
 \endproof

 We now examine what permutations belong to $\Lcal(S_{\beta})$. Due to Lemma~\ref{permconc}, it is sufficient to consider the case when $\beta$ has only one part.

 \begin{lemma} \label{catalanpermutation}
For a positive integer $b$, let $w$ be a permutation in $\Lcal({\bm 2}\times {\bm b})$. Then for each $k$,
\begin{equation} \label{catalanword}
w_k(1)-w_k(-1)\ge 0,
\end{equation}
where $w_k(1)$ and $w_k(-1)$ denote the numbers of $w_i \le b$ and $w_i> b$ for $1\le i \le k$, respectively.  
\end{lemma}

\proof
Let $\sigma:{\bm 2}\times {\bm b}\to \{1',\ldots, (2b)' \}$ be the linear extension identified with $w$.  Given $k$, set $(x,y)=\sigma^{-1}(k')$.

Suppose that $x=1$.
Let $\sigma(1,z)=i'_z$ for $1\le z\le b$.
Since $\sigma$ is an order-preserving bijection,
$$
i'_z \le k'  \text{ if and only if } z\le y.
$$
In addition, for all $z\le y$,
$$
w_{i_z}=\nu(\sigma^{-1}(i_{z}'))=\nu(1,z)=z \le b,
$$
where the last equality follows from the definition of $\nu$.  Thus $w_{k}(1)=y$.

Also, let $\sigma(2, z )=j'_z$ for $ 1 \le z\le b$. Since $\sigma$ is an order-preserving bijection,
$$
j'_z \ge  k'  \text{ if } z\ge y.
$$
And, for all $z\ge y$,
$$
w_{j_z}=\nu(\sigma^{-1}(j_{z}'))=\nu(2, z)=b+z > b,
$$
where the last equality follows from the definition of $\nu$.  Thus the number of $w_j>b$ for $j>k$ is at least $b-y+1$, so consequently,  $w_{k}(-1)\le y -1$.  Therefore,  \eqref{catalanword} is verified.

For $x=2$, it can be proved in a similar way. So we omit the details.
\endproof

Lemma~\ref{catalanpermutation} shows that permutations in $\Lcal({\bf 2}\times {\bm b})$ are Catalan permutations. In the following lemma, we will identify decent positions of such permutations.

\begin{lemma} \label{descent}
For a positive integer $b$, let $w$ be a permutation in $\Lcal({\bm 2}\times {\bm b})$. If $k$ is a decent position in $w$, then $w_{k}>b\ge w_{k+1}$.
\end{lemma}

\proof
Since both $\sigma$ and $\nu$ are both order-preserving,  as we seen in the proof of Lemma~\ref{catalanpermutation},
if $w_{i_z}=z$, then $i_1<\cdots <i_b$,  and if  $w_{j_z}=b+z$, then $j_1<\cdots <j_b$.
%
Thus if $k$ is a decent position, then $w_k>b \ge w_{k+1}$.
\endproof

\begin{example} \label{example3.2}
Given the poset $S_{(2,3,1,2)}$ labeled by $\nu$:
$$\nu=
\begin{array}{cccccccc}
1&2& & &  &  &  &  \\
3&4&5&6& 7&  &  &  \\
 & &8&9&10&11&  &  \\
 & & & &  &12&13&14 \\
 & & & &  &  &15&16
\end{array},
$$
let us consider the following linear extension $\sigma$:
$$
\sigma=
\begin{array}{cccccccc}
1'&3'& & &  &  &  &  \\
2'&4'&5'&7'& 8'&  &  &  \\
 & &6'&9'&10'&11'&  &  \\
 & & & &  &12'&13'&14' \\
 & & & &  &  &15'&16'
\end{array}.
$$
Then the permutation $w\in \Lcal(S_{(2,3,1,2)})$ associated with $\sigma$ is as follows:
\begin{align*}
w&= \nu(\sigma^{-1}(1'))\,\nu(\sigma^{-1}(2'))\,\cdots \,\nu(\sigma^{-1}(16'))\\
&=1\,3\,2\, 4\, 5\, 8\, 6\, 7\, 9\, 10\,11\, 12\, 13\, 14\, 15\, 16. 
\end{align*}

Meanwhile, a function from $S_{(2,3,1,2)}$ to $\setN$ defined by
$$
\begin{array}{cccccccc}
8&6& & & & & &  \\
7&6&6&6&5& & &  \\
 & &6&6&4&3& &  \\
 & & & & &1&1&0 \\
 & & & & & &0&0
\end{array}
$$
is an $S_{(2,3,1,2)}$-partition of $65$.
\end{example}

We now recall a result from \cite[Theorem 3.15.7]{sta1} and \cite[p. 381]{sta2}.

\begin{proposition} \label{prop:PP}
For a composition $\beta$ of $d$,
\begin{equation*}
\sum_{\pi \in \Acal(S_{\beta})} q^{|\pi|}
= \frac{1}{(q;q)_{2d}} \sum_{w\in \Lcal(S_{\beta})} q^{\maj(w)},
\end{equation*}
where $|\pi |$ denotes the sum of parts of $\pi$.
\end{proposition}

\subsection{Linear extensions and Dyck paths}

We now relate linear extensions of the poset $S_{\beta}$ to Dyck paths.

Let $\Ecal_{s}(x_1,\ldots,x_{r})$ be the set of Dyck paths from $(0,0)$ to $(2s,0)$ with $r$ marked returns at $(2x_1,0),\ldots,(2x_{r},0)$.
For $r=0$, $\Ecal_{s}(x_1,\ldots,x_{r})$ is the set of Dyck paths from $(0,0)$ to $(2s,0)$ with no marked returns, so  we define it to be $E^{(0)}_s$.  

\begin{proposition} \label{prop:maj}
For a composition $\beta=(b_1,\ldots,b_m)$ of $d$,
$$
\sum_{w\in \Lcal(S_{\beta})} q^{\maj(w)}
= \sum_{D \in \Ecal_{d}(r_1,\ldots,r_{m-1})} q^{\maj(D)-2(r_1+\cdots+r_{m-1})}.
$$
Here, $r_i=r_i(\beta)$ is defined in Definition~\ref{def:ri}. If $m=1$, then $r_1+\cdots +r_{m-1}$ will be the empty sum and defined to be $0$.
\end{proposition}
\begin{proof}

For a permutation $w\in\Lcal(S_{\beta})$, we first rewrite it as in \eqref{permutationw}.

For each permutation $w^{(i)}$ in $\Lcal({\bf 2}\times {\bm b_{\bm i}})$, we assign an up-step or down-step to each $w^{(i)}_j$, $1\le j\le 2b_i$,  as follows:
$$
w^{(i)}_j \longrightarrow
\begin{cases}
{\rm u}, \quad&  \text{ if $w^{(i)}_j \le b_i$},\\
{\rm d}, \quad& \text{ if $w^{(i)}_j >b_i$},
\end{cases}
$$
which we see from Lemma~\ref{catalanpermutation}   yields a Dyck path $D^{(i)}$ in $\Ecal_{b_i}^{(0)}$. Also, it follows from Lemma~\ref{descent} that the set of descent positions of $w^{(i)}$ is equal to the set of valley positions of $D^{(i)}$, namely
\begin{equation*}
\maj(w^{(i)})=\maj(D^{(i)}).
\end{equation*}
Thus, we have
$$
\sum_{w^{(i)} \in \Lcal({\bm 2}\times{\bm b_{\bm i} })} q^{\maj(w^{(i)})}
=\sum_{D^{(i)} \in \Ecal_{b_i}^{(0)}} q^{\maj(D^{(i)})}.
$$

By concatenating the paths $D^{(1)}, D^{(2)}, \ldots, D^{(m)}$, we obtain a path
$D \in \Ecal_d(r_1,\ldots,r_{m-1})$ such that
$$D=D^{(1)} \,|\, D^{(2)}\,|\, \cdots \,|\,D^{(m)},$$
where $|$ indicates  where marked returns occur.

We now check $\maj(w)$ and $\maj(D)$.  Since the concatenations of the permutations $w^{(i)}$ and the paths $D^{(i)}$ are merely shifting the positions of permutations and paths by the same units $2b_i$,  each descent position and its corresponding valley position still match after the concatenations. However, $D$ has valleys between $D^{(i)}$ and $D^{(i+1)}$ for $1\le i \le m-1$; while no descent occurs between $w^{(i)}$ and $w^{(i+1)}$. Therefore,
\begin{equation*}
\maj(w)=  \maj(D)- 2(r_1+\cdots r_{m-1}),
\end{equation*}
which completes the proof.
\end{proof}

\begin{example}
Let us recall the poset $S_{(2,3,1,2)}$ with the labeling $\nu$ and the linear extension $\sigma$ identified with the permutation $w=1\,3\,2\, 4\, 5\, 8\, 6\, 7\, 9\, 10\,11\, 12\, 13\, 14\, 15\, 1$   in Example~\ref{example3.2}.

Then
$$w^{(1)}=1324, \; w^{(2)}=142356, \; w^{(3)}=12,  \; w^{(4)}=1234,
$$
which yields
$$
D^{(1)}=\rm{udud},\;   D^{(2)}=\rm uduudd,\;  D^{(3)}=\rm ud, \;  D^{(4)}=\rm uudd,
$$
Thus the corresponding Dyck path $D\in \Ecal_{8}(2,3,1,2)$ is
$$
D=
D^{(1)} \,|\,D^{(2)} \,|\,D^{(3)} \,|\,D^{(4)}
=
{\rm u}\,{\rm d}\,{\rm u}\, {\rm d}\,|\, {\rm u}\, {\rm d}\, {\rm u}\, {\rm u}\, {\rm d}\, {\rm d}\,|\,{\rm u}\, {\rm d}\,|\, {\rm u}\, {\rm u}\, {\rm d}\, {\rm d}.
$$
Note that  $\maj(w)=2+6=8$ (between $3$ and $2$ and between $8$ and $6$ in $w$) and
$$\maj(D)-2(r_1+r_2+r_3)=(2+4+6+10+12)-2(2+5+6)=8=\maj(w).$$
\end{example}

\section{Proof of Theorem~\ref{thm:main}}\label{sec4}
Let $\lambda$ be  a Frobenius symbol with $d$ columns. Since each row of $\lambda$ is strictly decreasing, by subtracting $d-1$, $d-2,\ldots, 1$ and $0$ from the largest entry in each row, we obtain a two rowed array $\mu$ with weakly decreasing entries  in each row.  Since we subtracted the same quantity from the entries in each column, the parity of the column will remain the same.
We will call $\mu$ a Frobenius array.
We denote by $|\mu|$  the sum of the entries in $\mu$. Since $\lambda$ has $d$ columns, we see that
$|\mu|=|\lambda|-d^2$.


For  a composition $\beta=(b_1,b_2,\ldots,b_m)$ of $d$,
let $\Fcal(\beta)$ be the set of Frobenius arrays $\mu$ with $d$ columns and $m$ parity blocks $B_1,B_2,\ldots,B_m$, where $B_m$ is positive and each block $B_i$ has $b_i$ columns.

%

We now rearrange the entries of $\mu$ in $\Fcal(\beta)$ to form an $S_{\beta}$-partition $\gamma$ as follows.
\begin{itemize}
\item First, we interchange the top and bottom entries in each negative block of $\mu$.  
We call the resulting array $\hat{\mu}$.

\item  We then assign the $(i,j)$-th entry $\hat{\mu}_{i,j}$ of $\hat{\mu}$ to the $(i+ l, j)$-th entry of  $S_{\beta}$, where $r_{l}<j\le r_{l+1}$.
\end{itemize}
Clearly, $|\gamma|$, the sum of entries in $\gamma$, equals  $|\mu|$.

\begin{example}\label{ex4.1}
Let us consider
$$
\mu=\left( \begin{array}{cc|ccc|c|cc}
9&8&8&8&7&2&2&1\\
10&8&7&7&5&4&0&0
\end{array} \right).
$$
By interchanging the first and  second rows in each of the negative blocks,  we obtain
$$\hat{\mu}=
\left( \begin{array}{cc|ccc|c|cc}
10&8&8&8&7&4&2&1\\
9&8&7&7&5&2&0&0
\end{array} \right).
$$
By assigning the entries of $\hat{\mu}$ to the poset $S_{(2,3,1,2)}$, we obtain
$$
\gamma=\begin{array}{cccccccc}
10&8& & & & & &  \\
9 &8&8&8&7& & &  \\
 & &7&7&5&4& &  \\
 & & & & &2&2&1 \\
 & & & & & &0&0
\end{array}.
$$
\end{example}

We now prove that $\gamma$ is an $S_{\beta}$-partition.
\begin{lemma} \label{lem4.2}
For a Frobenius array $\mu \in \Fcal(\beta)$,  $\gamma$ is an $S_{\beta}$-partition, i.e.,  the entries of $\gamma$ in each row are decreasing and the entries in each column are decreasing.  In addition, the entries in each column that were in the positive blocks in $\mu$ are strictly decreasing.
\end{lemma}

\proof
It is clear from the construction of $\gamma$ that its underlying poset is $S_{\beta}$.

Note that for each column of $\mu$  the top entry is less than or equal to the bottom entry if they are in a negative block. So, by interchanging them, the entries are column-wise decreasing.  Also, top entry is greater than the bottom if they are in a positive block. Thus we get the assertion about the columns of $\gamma$.

We now prove the assertion about the rows of $\gamma$. Since the entries in the first row are from only the first block of $\mu$,
they are clearly decreasing.  So, we will check the other rows. Note that the entries in row $i$ are from the $(i-1)$-st and $i$-th blocks  of $\mu$.   Let $x_1$ and $y_1$ be the entries in the last column of the $(i-1)$-st block, and $x_2$ and $y_2$ be the entries in the first column of the $i$-th block of $\mu$, namely
$$
\mu=\left( \begin{array}{cc|cc}
\cdots & x_{1} & x_2 &\cdots \\
\cdots & y_1 & y_2 &\cdots
\end{array} \right).
$$
First, assume that the $(i-1)$-st and $i$th blocks are positive and negative, respectively.  Then, $\gamma$ will be as follows:
$$
\gamma = \; \begin{array}{cccc}
\cdots & x_{1} & & \\
\cdots & y_1 & y_2 &\cdots \\
& &  x_2 &\cdots
\end{array} \;
$$
Since $\mu$ is a Frobenius array,  it is clear that  $y_1\ge y_2$.  Thus the $i$-th row of $\gamma$ to which $y_1$ and $ y_2$  belong is decreasing.  We can similarly prove the case when the $(i-1)$-st and $i$-th blocks are negative and positive, so  we omit the details.
\endproof

We now construct an $S_{\beta}$-partition $\pi$ from $\gamma$ as follows:
\begin{itemize}
\item Subtract $\lfloor\frac{m+2-i}{2}\rfloor$ from each entry in the $i$-th row of $\gamma$.
We call the resulting array $\pi$.
\end{itemize}

Set $b_0=b_{m+1}=0$.  Since there are $m+1$ rows in $\gamma$ and there are $b_{i-1}+b_i$ entries in each row $i$, the quantity we subtracted from $\gamma$ is
\begin{align*}
\sum_{i=1}^{m+1} \left \lfloor \frac{m+2-i}{2} \right \rfloor (b_{i-1}+b_{i}) &= \sum_{i=2}^{m+1} \left \lfloor \frac{m+2-i}{2} \right \rfloor b_{i-1}+  \sum_{i=1}^{m} \left \lfloor \frac{m+2-i}{2} \right \rfloor  b_{i}\\
&= \sum_{i=1}^{m} \left( \left \lfloor \frac{m+1-i}{2} \right \rfloor  + \left \lfloor \frac{m+2-i}{2} \right \rfloor\right)  b_{i}\\
&= \sum_{i=1}^{m} (m+1-i)b_i\\
&=r_1+r_2+\cdots+r_{m-1}+r_m.
\end{align*}
Thus,
\begin{equation} \label{eq:gam-pi}
|\gamma |-|\pi |=r_1+r_2+\cdots+r_{m-1}+r_m .
\end{equation}


\begin{example}
Let us consider $\gamma$ from Example~\ref{ex4.1}:
$$
\gamma=\begin{array}{cccccccc}
10&8& & & & & &  \\
9 &8&8&8&7& & &  \\
 & &7&7&5&4& &  \\
 & & & & &2&2&1 \\
 & & & & & &0&0
\end{array}.
$$
Then
$$
\pi=\begin{array}{cccccccc}
8&6& & & & & &  \\
7 &6&6&6&5& & &  \\
 & &6&6&4&3& &  \\
 & & & & &1&1&0 \\
 & & & & & &0&0
\end{array}.
$$
\end{example}

\begin{lemma} \label{lem:gam-pi}
For each $\gamma$ from Lemma \ref{lem4.2}, $\pi$ is an $S_{\beta}$-partition.
\end{lemma}

\proof
Let us denote the entry of $\gamma$ (resp. $\pi$) in row $i$ from the bottom and column $j$ by $\gamma_{-i,j}$ (resp. $\pi_{-i,j}$).
 Note that the last parity block of  $\mu\in \Fcal(\beta)$ was positive. Thus,
 \begin{equation*}
 \gamma_{-(i+1),j}-\gamma_{-i,j} \begin{cases} >0 & \text{ if $i$ is odd},\\
 \ge 0 & \text{ if $i$ is even}. \end{cases}
 \end{equation*}
Thus, the subtraction of $\lfloor \frac{i}{2}\rfloor$ from each entry $\gamma_{-i,j}$ will result in
\begin{equation*}
 \pi_{-(i+1),j}-\pi_{-i,j} \ge 0,
 \end{equation*}
 which shows that $\pi$ is an $S_{\beta}$-partition.
 \endproof

We are now ready to prove Theorem~\ref{thm:main}.  Let us recall that $a_m^{+}(n;d)$ counts the number of partitions $\lambda$ of $n$ with $d$ columns and $m$ parity blocks, where the last block is positive.
Let $\cp(d,m)$ be the set of compositions of $d$ with $m$ parts. Then
{\allowdisplaybreaks
\begin{align*}
\sum_{n\ge 0} a^+_{m}(n;d)q^n
&= \sum_{\beta \in \cp(d,m)}  q^{d^2}
\left(\sum_{\mu\in \Fcal(\beta)} q^{|\mu|} \right)
\tag{$|\lambda|-|\mu|=d^2$}\\
&= \sum_{\beta \in \cp(d,m)} q^{d^2+(r_1+\cdots+r_{m-1}+r_m)}
\sum_{\pi \in \Acal(S_{\beta})} q^{|\pi|}
\tag{Lemma~\ref{lem:gam-pi} and \eqref{eq:gam-pi}}\\
&= q^{d^2+d} \sum_{\beta \in \cp(d,m)}
\frac{q^{r_1+\cdots+r_{m-1}}}{(q;q)_{2d}}
\sum_{w\in \Lcal(S_{\beta})} q^{\maj(w)}
\tag{$r_m=d$ and Proposition~\ref{prop:PP}}\\
&= \frac{q^{d^2+d}}{(q;q)_{2d}} \sum_{\beta \in \cp(d,m)}
\left(\sum_{D \in \Ecal_{d}(r_1,\ldots,r_{m-1})} q^{\maj(D)-(r_1+\cdots+r_{m-1})} \right) \tag{Proposition~\ref{prop:maj}}\\
&= \frac{q^{d^2+d}}{(q;q)_{2d}}
\sum_{D \in \Ecal_{d}^{(m-1)}} q^{\vmr(D)}
\tag{Definition of $\vmr(D)$ and $\Ecal_{s}^{(r)}$}  \\
&= \frac{q^{d^2+d}}{(q;q)_{2d}} \,
q^{\binom{m}{2}}\frac{1-q^m}{1-q^d}\qbin{2d}{d+m}.
\tag{Corollary~\ref{coro:maE}}
\end{align*}
}
Thus \eqref{eq:mainP} in Theorem~\ref{thm:main} is proven.

Similarly, we can prove the generating function for  $a^-_{m}(n;d)$ in Theorem~\ref{thm:main}.  The only difference is that for each composition $\beta$ of $d$, we need to subtract $\lfloor\frac{i-1}{2}\rfloor$ from each entry in the last $i$-th row of $\gamma$ to construct $\pi$. Then,
\begin{equation*}
|\gamma |-|\pi |=(m-1) b_1+(m-2)b_2+\cdots+ 2 b_{m-2}+ b_{m-1}=r_1+r_2+\cdots+r_{m-1}.
\end{equation*}
Since $r_m=d$, \eqref{eq:mainN} in Theorem~\ref{thm:main} holds.

\begin{example}
From Theorem~\ref{thm:main}, 
\begin{align*}
a^+_{2}(15;3)
&=[q^{15}]\frac{q^{3^2+3+\binom{2}{2}}}{(q;q)_{6}}\frac{1-q^2}{1-q^3}\qbin{6}{3+2}\\
&=[q^2]\frac{1}{(1-q)^2 (1-q^3)^2(1-q^4)(1-q^5)}=3,
\end{align*}
where $[q^{j}] \,g(q)$ denotes the coefficient of $q^{j}$ in $g(q)$.
Indeed there are  $3$ partitions of $15$ with $3$ columns and $2$ parity blocks, where the last block is positive:
$$
\left( \begin{array}{c|cc}
3 &2&1\\5 &1&0
\end{array} \right),
\quad
\left( \begin{array}{c|cc}
4&2&1\\4 &1&0
\end{array} \right),
\quad
\left( \begin{array}{cc|c}
3&2&1\\4&2&0
\end{array} \right).
$$
\end{example}

\section{Further results} \label{sec5}

We first recall the definition of singular overpartitions with dotted blocks from \cite{sy16}.
Let $m$ be a positive integer.  A singular overpartition with exactly $m$ dotted blocks and the last dotted block negative is a partition whose parity blocks start as
\begin{align*}
\underbrace{NPNP\cdots NPN}_{m} \cdots \quad \text{ or } \quad \underbrace{PNPN\cdots NPN}_{m+1} \cdots
\end{align*}
if $m$ is odd; and
\begin{align*}
\underbrace{PNPN\cdots NPN}_{m} \cdots \quad \text{ or } \quad \underbrace{NPNP \cdots NPN}_{m+1} \cdots
\end{align*}
if $m$ is even. Here and in the sequel, $P$ and $N$ stand for a positive parity block and negative parity block, respectively.

Also, a singular overpartition with exactly $m$ dotted blocks and the last dotted block positive  is a partition whose parity blocks start as
\begin{align*}
\underbrace{PNPN\cdots PNP}_{m}  \cdots \quad \text{ or } \quad \underbrace{NPNP\cdots PNP}_{m+1}  \cdots
\end{align*}
if $m$ is odd; and
\begin{align*}
\underbrace{NPNP\cdots PNP}_{m} \cdots \quad \text{ or } \quad \underbrace{PNPN\cdots PNP}_{m+1} \cdots
\end{align*}
if $m$ is even.

We now recall a theorem from \cite{sy16} for $(k,i)=(3,1)$.

\begin{theorem}[Theorem 3.1, \cite{sy16}]\label{thm5.1}
Let $m$ be a positive integer.
\begin{enumerate}
\item The number of singular overpartitions of $n$ with $m$ dotted blocks and the last dotted block positive  equals  $p(n-(3m^2+m)/2)$.
\item The number of singular overpartitions of $n$ with $m$ dotted blocks and the last dotted block negative  equals  $p(n-(3m^2-m)/2)$.
\end{enumerate}
\end{theorem}

\subsection{Proof of Theorem~\ref{thm1.2}}
We will prove the following theorem that is equivalent to Theorem~\ref{thm1.2}.

\begin{theorem} \label{thm4.1} 
For $m\ge 1$,
\begin{equation*}
a^+_m(n)=(-1)^{m-1}\sum_{l=0}^{m-1}(-1)^{l}
\left( p\left(n-\frac{3l^2+l}{2}\right)-p\left(n-\frac{3(l+1)^2-(l+1)}{2}\right) \right)
\end{equation*}
and
\begin{align*}
a^-_m(n)
&= (-1)^{m-1}\sum_{l=0}^{m-1}(-1)^l
\left( p\left(n-\frac{3l^2-l}{2}\right)-p\left(n-\frac{3(l+1)^2+(l+1)}{2}\right) \right).
\end{align*}
\end{theorem}

\proof
We first prove the formula for $a^+_m(n)$ by induction on $m$. First, let $m=1$. By the definition of singular overpartitions and Theorem~\ref{thm5.1},  we know that $p(n-1)$ counts the number of partitions of $n$ whose parity blocks start as $N$ or $PN$. Thus $p(n)-p(n-1)$ counts the number of partitions of $n$ that have exactly one parity block $P$, from which it follows that
\begin{equation*}
a^+_1(n)=p(n)-p(n-1).
\end{equation*}
We now assume that for $m\ge 1$, the statement holds true.

By the definition of singular overpartitions and Theorem~\ref{thm5.1}, we know that  $p(n-(3m^2-m)/2)$ counts the number of partitions of $n$ whose parity blocks start as
\begin{align*}
\underbrace{NPNP\cdots NPN}_{m} \cdots \quad \text{ or } \quad \underbrace{PNPN\cdots NPN}_{m+1} \cdots
\end{align*}
if $m$ is odd; and
\begin{align*}
\underbrace{PNPN\cdots NPN}_{m} \cdots \quad \text{ or } \quad \underbrace{NPNP \cdots NPN}_{m+1} \cdots
\end{align*}
if $m$ is even. Also, by  the definition of singular overpartitions and Theorem~\ref{thm5.1},  we know that $p(n-(3m^2+m)/2)$ counts the number of singular overpartitions of $n$ with $m$ dotted blocks and last dotted block positive, i.e., the number of partitions of $n$ whose parity blocks start as
\begin{align*}
\underbrace{PNPN\cdots PNP}_{m}  \cdots \quad \text{ or } \quad \underbrace{NPNP\cdots PNP}_{m+1}  \cdots
\end{align*}
if $m$ is odd; and
\begin{align*}
\underbrace{NPNP\cdots PNP}_{m} \cdots \quad \text{ or } \quad \underbrace{PNPN\cdots PNP}_{m+1} \cdots
\end{align*}
if $m$ is even. Thus
\begin{align*}
p\left(n-\frac{3m^2+m}{2}\right)-p\left(n-\frac{3(m+1)^2-(m+1)}{2}\right)
\end{align*}
counts the number of partitions of $n$ whose parity blocks are exactly of the type
\begin{align*}
\underbrace{PNPN\cdots PNP}_{m} \quad \text{ or } \quad \underbrace{NPNP\cdots PNP}_{m+1}
\end{align*}
if $m$ is odd; and
\begin{align*}
\underbrace{NPNP\cdots PNP}_{m} \quad \text{ or } \quad \underbrace{PNPN \cdots PNP}_{m+1}
\end{align*}
if $m$ is even. Therefore,
\begin{align*}
&(-1)^{m}\sum_{l=0}^{m}(-1)^{l}
\left( p\left(n-\frac{3l^2+l}{2}\right)-p\left(n-\frac{3(l+1)^2-(l+1)}{2}\right) \right)\\
&=-a^+_{m}(n)+ \left( p\left(n-\frac{3m^2+m}{2}\right)-p\left(n-\frac{3(m+1)^2-(m+1)}{2}\right) \right)
\end{align*}
counts the number of partitions of $n$ whose parity blocks are exactly of the type
\begin{align*}
\underbrace{NPNP\cdots PNP}_{m+1}
\end{align*}
if $m$ is odd; and
\begin{align*}
\underbrace{PNPN \cdots PNP}_{m+1}
\end{align*}
if $m$ is even.  This completes the proof.

The proof of the formula for $a^-_m(n)$ is similar, so we omit the details.
\endproof

\subsection{Proof of Theorem~\ref{thm1.4}}

By the definition, we know that $a^{+}(n;d)$ counts the number of partitions of $n$ whose Frobenius symbols have exactly $d$ columns with  the last column positive.  Let $\lambda$ be such a Frobenius symbol. Then the last entry $h$ of the bottom row in $\lambda$ is the smallest. Also, the top entry in the last column is greater than $h$ because the parity of this column is positive.  We now subtract $h$ from each entry in the bottom row and $h+1$ in the top row. Then the resulting array will be a Frobenius symbol $\tilde{\lambda}$ of $n-(2h+1)d$ with the last entry of the bottom row being $0$. Then the generating function for such Frobenius symbols $\tilde{\lambda}$  is
\begin{equation*}
\frac{q^{d^2}}{(q;q)_{d-1} (q;q)_{d}}.
\end{equation*}
Since $|\lambda|-|\tilde{\lambda}|=(2h+1)d$, the generating function for $\lambda$ becomes
\begin{equation*}
\sum_{h=0}^\infty \frac{q^{d^2}}{(q;q)_{d-1} (q;q)_{d}} q^{(2h+1)d} =\frac{q^{d^2+d}}{(q;q)_{d-1} (q;q)_{d} (1-q^{2d})}.
\end{equation*}
Similarly, we can show the case when the last column is negative. We omit the details.

\begin{remark}
We make some notes on relationships between the sequences discussed in this paper including $a^-_m(n)$ and $a^+_m(n)$ as well as $a^-(n;d)$ and $a^{+}(n;d)$.
\begin{enumerate}
\item From Theorem~\ref{thm4.1}, we have
\begin{equation*}
a^-_m(n) - a^+_m(n) = p\left(n-\frac{3m^2-m}{2}\right)-p\left(n-\frac{3m^2+m}{2}\right).
\end{equation*}
\item
From Theorem~\ref{thm:main}, we have
\begin{equation*}
a^-_{m}(n;d) = a^+_{m}(n+d; d).
\end{equation*}
Thus we also have
\begin{equation*}
a^-(n;d) = a^+(n+d; d).
\end{equation*}
\item
From Theorem~\ref{thm1.4}, we have
\begin{equation*}
a^-(n;d) - a^+(n,d) = \sum_{j= 1}^\infty a(n-2dj+1;d-1),
\end{equation*}
where $a(n;d)$ is the number of partitions of $n$ whose Frobenius symbols have exactly $d$ columns.
\end{enumerate}
\end{remark}


\end{document}